\documentclass[3p,times]{elsarticle}
\usepackage{etex}
\usepackage{multicol}
\usepackage{enumerate}
\usepackage{amssymb} \usepackage{amsfonts} \usepackage{amsmath}
\usepackage{amsthm} \usepackage{epsfig} 
\usepackage{amscd} \usepackage[all]{xy}
\usepackage{wrapfig}
\usepackage{amsbsy}
\usepackage{array}
\usepackage{mdwlist}
\usepackage{mathrsfs}
\usepackage{subfig}
\usepackage[utf8x]{inputenc}
\usepackage[T1]{fontenc}

\newtheorem{theorem}{Theorem}[section]
\newtheorem{lemma}[theorem]{Lemma}

\newtheorem{corollary}[theorem]{Corollary}
\newtheorem{property}[theorem]{Property}

\theoremstyle{definition}
\newtheorem{ex}[theorem]{Example}

\theoremstyle{remark}
\newtheorem{remark}[theorem]{Remark}

\newcommand{\p} {\ensuremath {\mathbb{P}}}
\newcommand{\E} {\ensuremath {\mathbb{E}}}

\newcommand{\N} {\ensuremath {\mathbb{N}}}
\newcommand{\R} {\ensuremath {\mathbb{R}}}

\newcommand{\I} {\ensuremath {\mathbb{I}}}

\newcommand{\A} {\ensuremath {\mathscr{A}}}
\newcommand{\X} {\ensuremath {\mathscr{X}}}

\newcommand{\M} {\ensuremath {\mathscr{M}}}
\newcommand{\D} {\ensuremath {\mathscr{D}}}

\newcommand{\mo} {\ensuremath {\mathscr{P}}}

\newcommand{\Qi} {\ensuremath {\mathscr{Q}}}
\newcommand{\bnu} {\ensuremath {\bar{\nu}}}
\newcommand{\tnu} {\ensuremath {\tilde{\nu}}}
\newcommand{\ngamma} {\ensuremath {\gamma^{\nu}}}
\newcommand{\bgamma} {\ensuremath {\gamma^{\bnu_m}}}

\newcommand{\neta} {\ensuremath {\eta_{\nu}}}
\newcommand{\brho} {\ensuremath {\bar{\rho}}}

\begin{document}
\begin{frontmatter}

\title
{Asymptotic equivalence of jumps Lévy processes and their discrete counterpart}

\author[LJK]{Pierre \'Etor\'e}
\author[LJK]{Sana Louhichi}
\author[LJK,corr]{Ester Mariucci}
\address[LJK]{\it Laboratoire LJK, Universit\'e Joseph Fourier UMR 5224 51, Rue des Math\'ematiques, Saint Martin d'H\`eres BP 53 38041 Grenoble Cedex 09}
\address[corr]{Corresponding Author, Ester.Mariucci@imag.fr}
\date{}

\begin{abstract}
We establish the global asymptotic equivalence between a pure jumps Lévy process $\{X_t\}$ on the time interval $[0,T]$ with unknown Lévy measure $\nu$ belonging to a non-parametric
class and the observation of $2m^2$ Poisson independent random variables with parameters linked with the Lévy measure $\nu$. The equivalence result is asymptotic as $m$ 
tends to infinity.
The time $T$ is kept fixed and the sample path is continuously observed. This result justifies the idea that, from a statistical point of view, knowing how many jumps fall into a grid of intervals gives asymptotically the same amount of information as observing $\{X_t\}$.
\end{abstract}

\end{frontmatter}

\section{Introduction}

In recent years, the Le Cam theory on the asymptotic equivalence between statistical models has aroused great interest and a large number of works
has been published on this subject. Roughly speaking, asymptotic equivalence means that any statistical inference procedure can be transferred from one experiment to the
 other in such a way that the asymptotic risk remains the same, at least for bounded loss functions. 
One can use this property in order to obtain asymptotic results working in a simpler but equivalent setting. 
For the basic concepts and a detailed description of the notion of asymptotic equivalence, we refer to \cite{LeCam,LC2000}. 
A short review on this topic will be given in Section \ref{subsec:lecam}.

\vspace{0.2cm}
The main result of this paper states that the observation of a pure jumps Lévy process on the time interval $[0,T]$ with unknown Lévy measure $\nu$ 
belonging to a non-parametric class $\M$ is  asymptotically equivalent to observe certain independent Poisson random variables whose parameters
are related to the jump measure $\nu$. 
More precisely, in Theorem \ref{teo1} we prove that, under some conditions on the class of admissible Lévy measures $\nu$, the two following experiments are equivalent.
The first one consists on observing the $2m^2$ independent Poisson random variables $R_{\pm\infty} = \mo(T\nu(J_{\pm \infty}))$, $R_{j,k} = \mo(T\nu(J_{j,k}))$, where
\begin{equation}\label{eqn:J}
 J_{-\infty}=\Big]-\infty,-m\Big],\quad J_{j,k}=\Big]k+\frac{j-1}{m},k+\frac{j}{m}\Big], \quad J_{\infty}=\Big]m,\infty\Big[,
\end{equation}
for $j=1,\dots,m$, $k=-m,\dots,m-1$, $(j,k)\neq (1,0), (m,-1)$.  The second one is obtained by continuously observing a trajectory $\{x_t\}$ of a pure jumps Lévy process with Lévy measure $\nu$ and finite variation $\eta_{\nu}:=\int y \nu(dy)<\infty$, i.e. with characteristic function given by
\begin{equation*}
  \E\big[e^{iux_t}\big]=\exp\Big(t\int_{\R}\big(e^{iuy}-1\big)\nu(dy)\Big),\quad \forall u\in \R.
\end{equation*}
In fact, we can relax the hypothesis of finite variation as we show in Theorem \ref{teo1}.

We already know that, in a probabilistic sense, the full jump structure of a Lévy process $\{X_t\}$ is described by the random measure associated with its jumps, i.e.
$$
\mu^X(\omega, \cdot) = \sum_{t\geq 0} \delta_{(t, \Delta X_t(\omega))}.
$$
In this work, via the Le Cam theory, we formalize the idea that the amount of information contained in $\nu$ is, from a statistical point of view, asymptotically equivalent to knowing how many jumps fall into the $2m^2$ intervals of the form $\{J_{j,k}\}$ (see Corollary \ref{cor}).

Our proof is based on the construction, for any given Lévy measure $\nu$ in the parameter space $\M$, of an adequate ``discrete'' approximation $\bnu_m$ of $\nu$. In this sense, the scheme of the proof is similar to that in Brown and Low \cite{BL}, but with the significant technical difference that, instead of a Girsanov type change of measure, we make use of an Esscher type change of measure. This adds further complications, which we bypass by considering as a discrete counterpart $\bar{\nu}_m$ a measure with constant density on $J_{j,k}$ and such that $\bar\nu_m(J_{j,k}) = \nu(J_{j,k})$.

The paper is organized as follows. Sections \ref{subsec:lecam} to \ref{defexp} fix assumptions and notations while the main results are given in Section \ref{main results}.
Some examples can be found in Section \ref{sec:ex}.
The proofs are postponed to Section \ref{sec:proofs} and, in part, to the Appendix.
\section{Statement of the main results}\label{sec2}
\subsection{Some properties of the Le Cam $\Delta$-distance} \label{subsec:lecam}
The concept of asymptotic equivalence that we shall adopt in this paper is tightly related to the Le Cam $\Delta$-distance between statistical experiments. 
A \emph{statistical model} is a triplet $\mo_j=(\X_j,\A_j,\{P_{j,\theta}; \theta\in\Theta\})$ where $\{P_{j,\theta}; \theta\in\Theta\}$ 
is a family of probability distributions all defined on the same $\sigma$-field $\A_j$ over the \emph{sample space} $\X_j$ and $\Theta$ is the \emph{parameter space}.
The \emph{deficiency} $\delta(\mo_1,\mo_2)$ of $\mo_1$
with respect to $\mo_2$ quantifies ``how much information we lose'' by using $\mo_1$ instead of $\mo_2$ and is defined as
$\delta(\mo_1,\mo_2)=\inf_K\sup_{\theta\in \Theta}||KP_{1,\theta}-P_{2,\theta}||_{TV},$
 where TV stands for ``total variation'' and the infimum is taken over all ``transitions'' $K$ (see \cite{LeCam}, page 18).
 In our setting, however, the general notion of ``transitions'' can be replaced with the notion of Markov kernels. Indeed, when the model $\mo_1$ is dominated 
and the sample space $(\X_2,\A_2)$ of the experiment $\mo_2$
is a Polish space, the infimum appearing on the definition of the deficiency $\delta$ can be taken over all Markov kernels $K$ on $\X_1\times \A_2$ 
(see \cite{N96}, Proposition 10.2), i.e.
$$\delta(\mo_1,\mo_2)=\inf_K\sup_{\theta\in \Theta}\sup_{A\in \A_2}\bigg|\int_{\X_1}K(x,A)P_{1,\theta}(dx)-P_{2,\theta}(A)\bigg|. $$
Closely associated with the notion of deficiency is the so called $\Delta$-distance, i.e. the pseudo metric defined by:
$$\Delta(\mo_1,\mo_2):=\max(\delta(\mo_1,\mo_2),\delta(\mo_2,\mo_1)).$$
Two sequences of statistical models $(\mo_{1}^n)_{n\in\N}$ and $(\mo_{2}^n)_{n\in\N}$ are called \emph{asymptotically equivalent}
if $\Delta(\mo_{1}^n,\mo_{2}^n)$ tends to zero as $n$ goes to infinity. 
There are various techniques to bound 
the $\Delta$-distance. In our context we will use the following two well-known properties (see \cite{LeCam}):
\begin{property}\label{pro1}
 Let $\mo_i=(\X,\A,\{P_{i,\theta}, \theta\in\Theta\})$, $i=1,2$, be two dominated statistical models with the same sample space $\X$ and parameter space $\Theta$. 
Let $\xi$ be a common dominating measure and $g_{i,\theta}=\frac{dP_{i,\theta}}{d\xi}$. Define
$$L_1(\mo_1,\mo_2)=\sup_{\theta\in\Theta}\int_{\X}\vert g_{1,\theta}(x)-g_{2,\theta}(x)\vert\xi(dx).$$
Then, $\Delta(\mo_1,\mo_2)\leq L_1(\mo_1,\mo_2).$
\end{property}
\begin{property}\label{sta}
 Let $\mo_i=(\X_i,\A_i,\{P_{i,\theta}, \theta\in\Theta\})$, $i=1,2$, be two statistical models and let $(\X_1,\A_1)$ be a Polish space. 
Let $S:\X_1\to\X_2$ be a sufficient statistics
such that the distribution of $S$ under $P_{1,\theta}$ is equal to $P_{2,\theta}$. Then $\Delta(\mo_1,\mo_2)=0$. 
\end{property}
\subsection{The parameter space}\label{par} 
In order to state our results we need some regularity assumptions on the elements $\nu$ belonging to the parameter space $\M$. We will require that: 
\begin{enumerate}[(M1)]
\item There exists a Lévy measure $\tnu$ such that $\nu \ll\tnu$ for all $\nu$ in $\M$ (we will write  $\rho^{\nu}$ for the density $\frac{d\nu}{d\tnu}$ ).
 \item  $\int_{\R} (\sqrt{\rho^{\nu}(y)}-1)^2 \tnu(dy) <\infty,$ for all $\nu$ in $\M$. 
\end{enumerate}   
Moreover, following the same principle as in \cite{BL}, we introduce a discretization of the measure $\nu$. To that aim define
\begin{equation}\label{brho}
\bar{\rho}_m^{\nu}(y)=\frac{\nu(J_{\pm\infty})}{\tilde{\nu}(J_{\pm\infty})}\ \textnormal{ in } J_{\pm\infty},\quad \bar{\rho}_m^{\nu}(y)=\frac{\nu(J_{j,k})}{\tilde{\nu}(J_{j,k})}\  \textnormal{ in } J_{j,k} \textnormal{ if }\ J_{j,k}\nsubseteq \Big]-\frac{1}{m},\frac{1}{m}\Big]\quad \textnormal{ and }\quad \bar{\rho}_m^{\nu}(y)=1\textnormal{ in } \Big]-\frac{1}{m},\frac{1}{m}\Big],
\end{equation}
where, for $j=1,\dots,m$ and $k=-m,\dots,m-1$, $J_{j,k}$ are defined as in \eqref{eqn:J}. Define one more condition as:
\begin{enumerate}[(M3)]
\item $\lim\limits_{m \to \infty}\sup\limits_{\nu\in \M}\int_{\R}\vert \rho^{\nu}(y)-\brho_m^{\nu}(y)\vert \tnu(dy)=0$,
\end{enumerate}
where we have denoted by $\bnu_m$ the measure having $\brho_m^{\nu}$ in \eqref{brho} as a density with respect to $\tnu$. For brevity's sake, 
in the sequel we will omit the symbol $\nu$ simply writing $\rho=\rho^{\nu}$ or $\brho_n=\brho_n^{\nu}$, when this causes no confusion.

 In the following, we will denote by $\M_{\tnu}$ a class of Lévy measures which we will always assume to satisfy (M1)--(M3). We will use the notation $\M'_{\tnu}$ for a class
 of Lévy measures which also satisfy condition (M4):
 \begin{enumerate}[(M4)]
\item $\sup\limits_{\nu\in \M_{\tnu}}\int_{|y|\leq 1} |y|\nu(dy)<\infty$ and $\int_{|y|\leq 1} |y|\tnu(dy)<\infty$.
\end{enumerate}

\subsection{Definition of the experiments}\label{defexp}
In the following, let $D=D([0,T],\R)$ be the space of mappings $\omega$ from $[0,T]$ into $\R$ that are right-continuous with left limits. Define the \emph{canonical process} $x:D\to D$ by 
$\forall \omega\in D,\quad x_t(\omega)=\omega_t,\;\;\forall t\in [0,T].$
 Let $\D$ be the smallest $\sigma$-algebra of parts of $D$ that makes $x_s$, $s$ in $[0,T]$, measurable. Further, for any 
 $t\in [0,T]$, let $\D_t$ be the smallest $\sigma$-algebra
that makes $x_s$, $s$ in $[0,t]$, measurable.
Let $\{X_t\}$ be a Lévy process defined on $(\Omega,\A,\p)$ having characteristic triplet $(\gamma,0,\nu)$, i.e. 
$$\E_{\p}\big[e^{iuX_t}\big]=\exp\bigg(t iu\gamma+t\int_{\R}\big(e^{iuy}-1-(iuy\I_{|y|\leq 1})\big)\nu(dy)\bigg),\quad \forall u\in \R,\ \forall t\geq 0.$$ It is well known that it induces a probability measure 
$P^{(\gamma,0,\nu)}$ 
on (D,\D) such that $\{x_t\}$ defined on $\big(D,\D,P^{(\gamma,0,\nu)}\big)$ is a Lévy process identical in law with 
$(\{X_t\},\p)$.
In the sequel we will denote by $\big(\{x_t\},P^{(\gamma,0,\nu)}\big)$ such a Lévy process, stressing the probability measure.

In the case where $\int_{|y|\leq 1}|y|\nu(dy)<\infty$, we set $\eta_{\nu}:=\int y \nu(dy)$. Let us remark that $(\{x_t\},P^{(\eta_{\nu},0,\nu)})$ is a pure jumps Lévy process. Moreover, if $\nu$ is a finite Lévy measure, then the process $(\{x_t\},P^{(\eta_{\nu},0,\nu)})$ is a compound Poisson process.

Let us now introduce the class of experiments that we shall consider. 
Recall from the introduction that we have defined $2m^2$ independent observations of the form (denoting $\mo(\cdot)$ the Poisson distribution)
\begin{equation*}
        R_{-\infty}\sim \mo\bigg(T\nu\Big(\big]-\infty,-m\big]\Big)\bigg),\quad
      R_{j,k}\sim \mo\bigg(T\nu\Big(\big]k+\frac{j-1}{m},k+\frac{j}{m}\big]\Big)\bigg), \quad
      R_{\infty}\sim \mo\bigg(T\nu\Big(\big]m,\infty\big[\Big)\bigg).
\end{equation*}
Let $Q^{m,R}_{\nu}$ be the law of $R = (R_{-\infty},\dots, R_{j,k},\dots,R_{\infty})$. Then, the first pair of statistical experiments is described by 
\begin{align}\label{modello1}
 \mo^{(\ngamma,0,\nu)}&=\Big(D,\D,\big\{P^{(\ngamma,0,\nu)}; \nu\in \M_{\tnu}\big\}\Big)\\
\label{modelloR}
 \Qi_m^R&=\Big(\N^{2m^2},\mathcal{P}(\N^{2m^2}),\big\{Q^{m,R}_{\nu}; \nu\in \M_{\tnu}\big\}\Big),
\end{align}
where $\gamma^{\nu}:=\int_{\vert y \vert \leq 1}y(\nu(dy)-\tnu(dy)).$
Remark that this quantity is finite thanks to Assumption (M2) (See \cite{sato}, Remark 33.3).

Recall that we denote by $\M'_{\tnu}:=\{\nu \in \M_{\tnu}: \nu \textnormal{ satisfies (M4)}\}$. Then, the second pair of statistical models we shall consider is
\begin{align}\label{modello2}
 \mo^{'(\neta,0,\nu)}&=\Big(D,\D,\big\{P^{(\neta,0,\nu)}; \nu\in \M'_{\tnu}\big\}\Big)\\
\label{modelloR'}
 \Qi_m^{' R}&=\Big(\N^{2m^2},\mathcal{P}(\N^{2m^2}),\big\{Q^{m,R}_{\nu}; \nu\in \M'_{\tnu}\big\}\Big).
\end{align}
\subsection{Main result}\label{main results}
Notations will be kept as in Section \ref{defexp}.
\begin{theorem}\label{teo1}
 The experiments $\mo^{(\ngamma,0,\nu)}$ and $\Qi^R_m$ 
are asymptotically equivalent, that is
\begin{equation}\label{eaR}
\lim_{m\to\infty} \Delta\big(\mo^{(\ngamma,0,\nu)},\Qi^R_m\big)=0.
\end{equation}
We also have: 
\begin{equation}\label{eaR'}
\lim_{m\to\infty} \Delta\big(\mo^{'(\neta,0,\nu)},\Qi^{' R}_m\big)=0.
\end{equation}
\begin{remark}
 The asymptotic equivalence in \eqref{eaR'} involves the parameter space $\M'_{\tnu}$ which is smaller than that in \eqref{eaR}, but this allows us to treat the model 
 $\mo^{'(\neta,0,\nu)}$, that is, the case of pure jumps Lévy processes.
\end{remark}
\end{theorem}
 Loosely speaking, the main interest in the Le Cam's asymptotic decision theory lies in the approximation of general statistical experiments by simpler ones. Adopting this point of view, we can reformulate Theorem \ref{teo1} as follows.
\begin{corollary}\label{cor}
So far as the study of $\nu$ is concerned, observing a Lévy process $\{X_t\}$ of 
characteristic triplet $(\ngamma,0,\nu)$ (or $(\neta,0,\nu)$) asymptotically gives the same amount of information as the \emph{a priori} coarser process 
$\big\{\sum_{t\leq T}\I_A(\Delta X_t)\big\}_{A\in \A^m}$, where $\mathscr{A}^m$ is the set defined by 
$\mathscr{A}^m=\Big\{J_{\pm \infty}, J_{j,k}; k=-m,\dots,m-1, j=1,\dots,m, (j,k)\neq (1,0),\ (m,-1)\Big\}$.
\end{corollary}

\begin{remark}
 Ideally, one would like to push the equivalence of Theorem \ref{teo1} one step further, to reduce to a Gaussian white noise experiment; in this way, one would have at hand a minimax estimator which could be carried over to the jumps model without loss of information. However, even though the $R_{j,k}$ variables are independent, they are not identically distributed, so that results such as \cite{N96} do not apply directly. A possible solution could be to use the proximity, in the Le Cam sense, between independent Poisson random variables non identically distributed and gaussian variables with fixed variance equal to 1. This will be the object of a future work. 
\end{remark}

\section{Examples}\label{sec:ex}
We now propose three different examples fitting in the hypotheses of Theorem \ref{teo1}. The first one is a case where the class of Lévy measures $\M'_{\tnu}$ consists of finite measures.
Specifically, it treats the class of compound Poisson processes with bounded intensities and uniformly Liptschtiz continuous densities. The last two examples concern infinite Lévy measures.
In the second one condition (M4) fails while the third one treats a class of tempered stable processes verifying (M4). Proofs can be found in \ref{ex}.

\begin{ex}\label{ex1}
Let $\tnu$ be a known finite Lévy measure and $L,K$ be fixed finite positive real numbers.
Consider
 \begin{equation*}
 \M_{\tnu}^{' L,K}=\Big\{\nu \textnormal{ Lévy measure }: \rho=\frac{d\nu}{d\tnu}  \textnormal{ exists, differentiable } \tnu \textnormal{-a.e.},
\textnormal{ $L$-Liptschtiz and }|\rho(0)|\leq K    \Big\}. 
\end{equation*}
Remark that the stated conditions on $\tnu$ and $\rho$ imply in particular that $\nu$ must be finite, i.e. $\nu(\R)<\infty$.
\end{ex}

\begin{ex}\label{ex2}
 Let $M$ and $\varepsilon$ be fixed positive real numbers. Consider the class of Lévy measures, defined by:
\begin{equation*}
 \M^{M,\varepsilon}=\Big\{\nu \textnormal{ Lévy measure }: \textnormal{ its density with respect to Lebesgue is } g(y)=e^{-\lambda y^2}y^{-2}
                         \textnormal{ where } \varepsilon\leq \lambda\leq M \Big\}.
\end{equation*}
\end{ex}
\begin{ex}\label{ex3}
Let $\alpha,\varepsilon, C_1, C_2$ be known positive real numbers such that $\alpha<1$ and let $M$ be a positive number. Consider the class 
of the tempered stable Lévy measures, that is
\begin{equation*}
 \M^{'\alpha,C_1,C_2,M,\varepsilon}=\Big\{\nu \textnormal{ Lévy measure }: \textnormal{ its density is }
\nonumber  g(y)=\frac{C_1}{|y|^{1+\alpha}}e^{-\lambda_1|y|}\I_{y<0}+\frac{C_2}{y^{1+\alpha}}e^{-\lambda_2y}\I_{y\geq0}
\textnormal{ where }  \varepsilon\leq \lambda_j\leq M,\ j=1,2\Big\}. 
\end{equation*}
 \end{ex}
 
\section{Proofs}\label{sec:proofs}
\subsection{Proof of Theorem \ref{teo1}}
One important ingredient of the proof is an Esscher type change of measure. 
Denote by $P|_{\D_t}$ the restriction of the probability $P$ to $\D_t$ and write $\nu\approx\tnu$ to indicate that the measure $\nu$ and $\tnu$ are equivalent.
The following result will be used in our proof.

\begin{theorem}[See \cite{sato}, Theorems 33.1--33.2]\label{teosato}
 Let $\big(\{x_t\},P^{(0,0,\tnu)}\big)$ and $\big(\{x_t\},P^{(\ngamma,0,\nu)}\big)$ be two Lévy processes on $\R$, where
\begin{equation}\label{gamma*}
 \ngamma:=\int_{\vert y \vert \leq 1}y(\nu-\tnu)(dy)
\end{equation}
is supposed to be finite. Then $P^{(\ngamma,0,\nu)}$ is locally equivalent to $P^{(0,0,\tnu)}$ if and only if $\nu\approx\tnu$ and the density $\frac{d\nu}{d\tnu}=\rho$ satisfies
\begin{equation}\label{Sato}
 \int(\sqrt{\rho(y)}-1)^2\tnu(dy)<\infty.
\end{equation}
When $P^{(\ngamma,0,\nu)}$ is locally equivalent to $P^{(0,0,\tnu)}$, the density is
$\frac{dP^{(\ngamma,0,\nu)}}{dP^{(0,0,\tnu)}}\Big|_{\D_t}(x)=\exp(U_t^{\rho}(x)),$
with
\begin{equation}\label{U}
 U_t^{\rho}(x)=\lim_{\varepsilon\to 0} \Big(\sum_{r\leq t}\ln \rho(\Delta x_r)\I_{\vert\Delta x_r\vert>\varepsilon}-
\int_{\vert y\vert > \varepsilon} t(\rho(y)-1)\tnu(dy)\Big),\quad P^{(0,0,\tnu)}\textnormal{-a.s.}
\end{equation}
The convergence in \eqref{U} is uniform in $t$ on any bounded interval, $P^{(0,0,\tnu)}$-a.s.
Besides, $U^{\rho}(x)$ defined by \eqref{U} is a Lévy process satisfying $\E_{P^{(0,0,\tnu)}}[e^{U_t^{\rho}(x)}]=1$, $\forall t\in [0,T]$.
\end{theorem}
Remark that the finiteness in \eqref{Sato} implies that in \eqref{gamma*}  (see \cite{sato}, Remark 33.3). 
The proof of Theorem \ref{teo1} is divided in three steps.

\textbf{STEP 1.} The task is to prove that: $\lim\limits_{m\to \infty}\Delta\big(\mo^{(\ngamma,0,\nu)},\mo^{(\bgamma,0,\bnu_m)}\big)=0$.\\
Recall that $\bnu_m$ is the Lévy measure defined in Section \ref{par} and observe that $\gamma^{\bnu_m}$ is finite thanks to Hypothesis \eqref{Sato} and the definition of $\brho_m$. 
We know, by Theorem \ref{teosato}, that the $L_1$ distance between $P^{(\ngamma,0,\nu)}$ and $P^{(\bgamma,0,\bnu_m)}$ is given by
\begin{equation*}
L_1\Big(P^{(\ngamma,0,\nu)},P^{(\bgamma,0,\bnu_m)}\Big)=\E_{P^{(0,0,\tnu)}}\bigg[\Big| \frac{dP^{(\ngamma,0,\nu)}}{dP^{(0,0,\tnu)}}(x)
-\frac{dP^{(\bgamma,0,\bnu_m)}}{dP^{(0,0,\tnu)}}(x)\Big|\bigg]
=\E_{P^{(\ngamma,0,\nu)}}\Big[\big| 1-\exp\big(U_T^{\brho_m}(x)-U_T^{\rho}(x)\big)\big|\Big], 
\end{equation*}
with $U_T^{\rho}(x)$ defined as in \eqref{U}. Introduce the quantity
$R_T^m(x):=\exp\big(U_T^{\brho_m}(x)-U_T^{\rho}(x)\big)$
and observe that, by definition, $
R_T^m(x)=\exp\bigg(\lim_{\varepsilon\to 0}\Big(\sum_{r\leq T}\ln \frac{d\bnu_m}{d\nu}(\Delta x_r)\I_{|\Delta x_r|> \varepsilon}-
T\int_{|y|>\varepsilon} \Big(\frac{d\bnu_m}{d\nu}(y)-1\Big)\nu(dy)\Big)\bigg),
$ $P^{(\ngamma,0,\nu)}$-a.s. By Lemma \ref{lemmadisc}, we get: $L_1\Big(P^{(\ngamma,0,\nu)},P^{(\bgamma,0,\bnu_m)}\Big)\leq 2\sinh\Big(T\int_{\R}\vert\rho(y)-\brho_m(y) \vert \tnu(dy)\Big).$
Thus, thanks to Assumption (M3), we have 
$\lim_{m\to\infty}\sup_{\nu\in\M_{\tnu}}L_1\Big(P^{(\ngamma,0,\nu)},P^{(\bgamma,0,\bnu_m)}\Big)=0.$
By Property \ref{pro1}, we conclude that the models $\mo^{(\ngamma,0,\nu)}$ and $\mo^{(\bgamma,0,\bnu_m)}$ are asymptotically equivalent as $m$ goes to infinity. 

\vspace{0.3cm}

\textbf{STEP 2.} The goal is to prove that: $\Delta\big(\mo^{(\bgamma,0,\bnu_m)},\Qi_m^R\big)=0$  for all $m$.\\
Consider the statistics $S:(D,\D)\to \big(\bar\N^{2m^2+2},\mathcal{P}(\bar\N^{2m^2+2})\big)$  defined by 
\begin{equation*}
S(x)=\bigg(N_T^{x;-\infty},N_T^{x;\,1,-m},\dots,N_T^{x;\,m,m-1},N_T^{x;\infty}\bigg)\quad \textnormal{with} \quad
 N_T^{x;\pm\infty}=\sum_{r\leq T}\I_{J_{\pm\infty}}(\Delta x_r),\quad N_T^{x;\,j,k}=\sum_{r\leq T}\I_{J_{j,k}}(\Delta x_r).
\end{equation*}
Applying Theorem \ref{teosato} to $\big(\{x_t\}, P^{(\bgamma,0,\bnu_m)}\big)$ and $(\{x_t\},P^{(0,0,\tnu)})$, we obtain that 
\begin{equation*}
\frac{d P^{(\bgamma,0,\bnu_m)}}{dP^{(0,0,\tnu)}}(x)=e^{U^{\brho_m}_T(x)}=\exp\Bigg(\ln\frac{\nu(J_{-\infty})}{\tilde{\nu}(J_{-\infty})} N_T^{x;-\infty}
+\ln\frac{\nu(J_{\infty})}{\tilde{\nu}(J_{\infty})} N_T^{x;\infty}
 +\sum_{\substack {j=1,\dots,m\\k=-m,\dots,m-1\\(j,k)\neq (1,0),(m,-1)}} \ln\frac{\nu(J_{j,k})}
{\tilde{\nu}(J_{j,k})} N_T^{x;\, j,k}
  +T\int_{\R} (\brho_m(y)-1)\tnu(dy)\Bigg).
\end{equation*}
Hence, by means of the Fisher factorization theorem, we conclude that $S$ is a sufficient statistics for $\{P^{(\bgamma,0,\bnu_m)}; \nu\in \M_{\tnu}\}$. 
Furthermore, under $P^{(\bgamma,0,\bnu_m)}$, the random variables $N_T^{x;j,k}$ (resp. $N_T^{x;-\infty}$ and $N_T^{x;\infty}$) have Poisson distributions 
with parameters $T\bar{\nu}_m\big(J_{j,k}\big)$ 
(resp. $T\bar{\nu}_m\big(J_{-\infty}\big)$ and $T\bar{\nu}_m\big(J_{\infty}\big)$), which, by the definition of $\bar{\nu}_m$, is
equal to $T\nu\big(J_{j,k}\big)$ (resp. $T\nu\big(J_{-\infty}\big)$, $T\nu\big(J_{\infty}\big)$). Moreover, since the J's intervals are disjoint, 
the random variables $N_T^{x,\cdot}$ are independent. 
It follows that the law of $S$ under $P^{(\bgamma,0,\bnu_m)}$ is 
$Q_{\nu}^{m,R}$. Then, by means of Property \ref{sta}, we get 
$\Delta(\mo^{(\bgamma,0,\bnu_m)}, \Qi_m^R)=0, \textnormal{ for all } m.$

\textbf{STEP 3.} The purpose is to prove that: if $\nu$ belongs to $\M'_{\tnu}$ then,
$\Delta(\mo^{(\ngamma,0,\nu)},\mo^{'(\neta,0,\nu)})=0.$
To that aim, consider the Markov kernels $\pi_1$, $\pi_2$ defined as follows
\begin{equation*}
\pi_1(x,A)=\I_{A}(x^d), \quad 
 \pi_2(x,A)=\I_{A}(x-\cdot \eta_{\tnu}),
\quad \forall x\in D, A \in \D,
\end{equation*}
where we have denoted by $x^d$ the discontinuous part of the trajectory $x$, i.e.
$\Delta x_r = x_r - \lim_{s \uparrow r} x_s,\ x_t^d=\sum_{r \leq t}\Delta x_r$ and by $x-\cdot \eta_{\tnu}$ the trajectory $x_t-t\eta_{\tnu}$, $t\in[0,T]$.
On the one hand we have:
\begin{equation*}
 \pi_1 P^{(\ngamma,0,\nu)}(A)=\int_D \pi_1(x,A)P^{(\ngamma,0,\nu)}(dx)=\int_D \I_A(x^d)P^{(\ngamma,0,\nu)}(dx)=P^{(\neta,0,\nu)}(A),
\end{equation*}
where in the last equality we have used the fact that, under $P^{(\ngamma,0,\nu)}$, $\{x_t^d\}$ is a Lévy process with characteristic triplet $(\neta,0,\nu)$ 
(see \cite{sato}, Theorem 19.3).
On the other hand:
\begin{equation*}
 \pi_2 P^{(\neta,0,\nu)}(A)=\int_D \pi_2(x,A)P^{(\neta,0,\nu)}(dx)=\int_D \I_A(x-\cdot\eta_{\tnu})P^{(\neta,0,\nu)}(dx)=P^{(\ngamma,0,\nu)}(A),
\end{equation*}
since, by definition, $\ngamma$ is equal to $\neta-\eta_{\tnu}$. The conclusion follows by the definition of the $\Delta$-distance.
\qed

\begin{proof}[Proof of Corollary \ref{cor}]
It is enough to note that, for all $A$ in $\A^m$, the random variable $\sum_{t\leq T}\I_{A}(\Delta x_t)$ has a Poisson distribution of parameter $T\nu(A)$ 
under both $P^{(\ngamma,0,\nu)}$ and $P^{(\neta,0,\nu)}$. The independence follows since the elements of $\A^m$ are disjoint.  
\end{proof}
\appendix
\section{Appendix}
\subsection{A technical lemma}
\begin{lemma}\label{lemmadisc}
The following limit
\begin{equation}\label{RT}
R_T^m(x):=\lim_{\varepsilon\to0}\bigg(\exp\Big(\sum_{r\leq T}\ln \frac{\bar{\rho}_m}{\rho}(\Delta x_r)\I_{\vert \Delta x_r\vert > \varepsilon}-
T\int_{\vert y \vert >\varepsilon}(\bnu_m-\nu)(dy)\Big)\bigg),
\end{equation}
with $\rho$, $\brho_m$, $\ngamma$ and $\bnu_m$ as in Section \ref{sec2}, exists uniformly in $t$ in any bounded interval, $P^{(\ngamma,0,\nu)}$-a.s. and
 \begin{equation}\label{eq:ipd}
  \E_{P^{(\ngamma,0,\nu)}}\Big[\big| 1-R_T^m(x)\big|\Big]\leq 2\sinh\bigg(T\int_{\R}\vert \rho(y)-\brho_m(y)\vert \tnu(dy)\bigg).
 \end{equation}
\end{lemma}
\begin{proof}
To prove the existence of the limit in \eqref{RT} we want to apply Theorem \ref{teosato}. To that aim just note that Assumption (M2) implies the finiteness of the integral 
$\int_{\R}\bigg(\sqrt{\frac{d\bnu_m}{d\nu}(y)}-1\bigg)^2\nu(dy).$
Indeed, integrating the last quantity separately over the intervals $\big[-\frac{1}{m},\frac{1}{m}\big]$, $\big]\frac{1}{m},\infty\big[$ and $\big]-\infty,-\frac{1}{m}\big[$, we obtain
\begin{align*}
\int_{\R}\bigg(\sqrt{\frac{d\bnu_m}{d\nu}(y)}-1\bigg)^2\nu(dy)&= \int_{-\frac{1}{m}}^{\frac{1}{m}}\big(1-\sqrt{\rho(y)}\big)^2\tnu(dy)+ \int_{\frac{1}{m}}^{\infty}\big(\sqrt{\brho_m(y)}-\sqrt{\rho(y)}\big)^2\tnu(dy) +\int_{-\infty}^{-\frac{1}{m}}\big(\sqrt{\brho_m(y)}-\sqrt{\rho(y)}\big)^2\tnu(dy)\\
&\leq \int_{-\frac{1}{m}}^{\frac{1}{m}}\big(1-\sqrt{\rho(y)}\big)^2\tnu(dy)+4\nu\bigg(\Big]\frac{1}{m},\infty\Big[\cup \Big]-\infty,-\frac{1}{m}\Big[\bigg),
\end{align*}
that is finite thanks to Assumption (M2) and the fact that $\nu$ is a Lévy measure (in the last inequality we have used the elementary inequality: for $a,b\geq 0$, $(\sqrt a-\sqrt b)^2\leq 2a+2b$).
 In order to simplify the notations let us write 
 $$A^{\pm}(x):=\lim_{\varepsilon\to 0}\bigg(\sum_{r\leq T}\ln f^{\pm}(\Delta x_r)\I_{|\Delta(x_r)|>\varepsilon}-
T\int_{|y|>\varepsilon}(f^{\mp}(y)-1)\nu(dy)\bigg)\quad\textnormal{ with } f^+= \Big(\frac{\bar{\rho}_m}{\rho}\Big)^{\I_{\brho_m\geq\rho}} \textnormal{ and } f^-=\Big( \frac{\brho_m}{\rho}\Big)^{\I_{\rho>\brho_m}},$$ 
so that $R_T^m(x)=\exp(A_+(x)+A_-(x)).$
Remark that,  for all $x,y$ in $\R$ we have: $\vert 1-e^{x+y} \vert \leq \frac{1+e^x}{2}\vert 1-e^y\vert+\frac{1+e^y}{2}\vert 1-e^x\vert.$ Then, using $A_+(x)\geq 0$ and $A_-(x)\leq 0$ we get:
\begin{align*}
 \E_{P^{(\ngamma,0,\nu)}}\big[\vert 1-R_T^m(x) \vert\big]=\E_{P^{(\ngamma,0,\nu)}}\big| 1-\exp(A^+(x)+A^-(x))\big| 
                              &\leq \E_{P^{(\ngamma,0,\nu)}}\bigg[\frac{1+e^{A^+(x)}}{2}\Big| 1-e^{A^-(x)}\Big|+\frac{1+e^{A^-(x)}}{2}\Big| 1-e^{A^+(x)}\Big|\bigg]\\
                              &=\E_{P^{(\ngamma,0,\nu)}}\Big[e^{A^+(x)}-e^{A^-(x)}\Big].
\end{align*}
In order to compute the last quantity we apply Theorem \ref{teosato} and the fact that both $A^+(x)$ and $A^-(x)$ have the same law under $P^{(\ngamma,0,\nu)}$ and 
$P^{(0,0,\nu)}$:
\begin{align*}
   \E_{P^{(\ngamma,0,\nu)}}\Big[e^{A^+(x)}-e^{A^-(x)}\Big] &=\exp\bigg(T\int_{\R}(f^+(y)-f^-(y))\nu(dy)\bigg)-\exp\bigg(T\int_{\R}(f^-(y)-f^+(y))\nu(dy)\bigg)\\
      &=2\sinh\bigg(T\int_{\R}(f^+(y)-f^-(y)\nu(dy)\bigg)\nonumber =2\sinh\bigg(T\int_{\R}\vert \rho(y)-\brho_m(y)\vert \tnu(dy)\bigg)\nonumber.
\end{align*}
\end{proof}
\subsection{Proofs of the examples}\label{ex}
\begin{proof}[Proof of example \ref{ex1}]
Assumption (M1) is obvious by construction, Assumptions (M2) and (M4) follow from the finiteness of the measures $\nu$ and $\tnu$ plus the inequality $(\sqrt{\rho}-1)^2\leq \rho+1$. Prior to considering Assumption (M3), we claim that:
\begin{equation}\label{L}
|\rho(y)-\bar{\rho}_m(y)|\leq \frac{1}{m}\bigg\|\frac{d\rho(y)}{dy}\bigg\|_{\infty} \quad \forall y\in \Big]-m,-\frac{1}{m}\Big]\cup\Big]\frac{1}{m},m\Big].
\end{equation}
In order to prove \eqref{L}, fix an interval $J_{j,k}$, $j=1,\dots,m$ and $k=-m,\dots,m-1$ and note that, by construction of $\brho_m$, there always exist $y_1$, $y_2$ verifying
$\brho_m(y)\leq \rho(y_1) \textnormal{ and }\brho_m(y)\geq \rho(y_2)\quad \forall y\in J_{j,k}.$
Thus, by the continuity of $\rho$, we conclude that, for all $y$ in $J_{j,k}$, there exists $\hat y$ such that $\brho_m(y)=\rho(\hat y)$. Then we apply the mean value theorem to bound $|\rho(y)-\brho_m(y)|$.
 Now, Assumption (M3) is a straightforward consequence of the inequality \eqref{L}.
Indeed, since $\Big\|\frac{d\rho(y)}{dy}\Big\|_{\infty}\leq L$ and $|\rho(0)|\leq K$, we have
\begin{align*}
\int_{\R}|\rho(y)-\brho_m(y)|\tnu(dy)&=\int\limits_{\frac{1}{m}\leq|y|< m}|\rho(y)-\brho_m(y)|\tnu(dy) +\int\limits_{|y|\geq m}|\rho(y)-\brho_m(y)|\tnu(dy)+\int\limits_{|y|< \frac{1}{m}}|\rho(y)-1|\tnu(dy)\\
&\leq \frac{L}{m}\tnu\bigg(\Big]-m,-\frac{1}{m}\Big]\cup\Big]\frac{1}{m},m\Big]\bigg)+ 2\int_{|y|\geq m}(K+L|y|)\tnu(dy)+\int_{|y|< \frac{1}{m}} (K+L|y|+1)\tnu(dy).
\end{align*} 
This quantity tends to zero, uniformly on $\nu$, as $m$ goes to infinity, since $\tnu$ is a finite Lévy measure such that $\int_{|y|\leq 1}|y|\tnu(dy)<\infty$.
\end{proof}
\begin{proof}[Proof of example \ref{ex2}]
Consider $\nu$ in $\M^{M,\varepsilon}$ and define a Lévy measure $\tnu$, absolutely continuous with respect to Lebesgue, whose density is $y^{-2}$. Consequently, 
$\rho(y)=\frac{d\nu}{d\tnu}(y)=e^{-\lambda y^2}$. Condition (M2) writes 
\begin{equation*}
 \int_{\R} (\sqrt {\rho(y)}-1)^2\tnu(dy)=2\int_0^{\infty} \big(e^{\frac{-\lambda y^2}{2}}-1\big)^2 y^{-2}dy
\leq 2\lambda^2\int_0^{1} \frac{y^2}{4}dy+4(e^{-\frac{\lambda}{2}}+1)\int_1^{\infty}\frac{dy}{y^2} < \infty.
\end{equation*}
To verify condition (M3), treat separately the integral $\int_0^{\infty} |\rho(y)-\brho_m(y)|\tnu(dy)$ over the intervals $\Big[0,\frac{1}{m}\Big]$, $\Big]\frac{1}{m},C\Big]$, 
$]C,m]$, $]m,\infty[$, where $C=\sqrt{\frac{1}{3 \varepsilon}}$ is chosen so that $\rho'$ is strictly decreasing in the interval $\Big]\frac{1}{m},\frac{\lceil Cm\rceil}{m}\Big]$ for $m$ big enough:
\begin{align*}
 &\int_0^{\frac{1}{m}} |\rho(y)-1|y^{-2}dy \leq \lambda \int_0^{\frac{1}{m}}\frac{y^2}{y^2}dy \leq \frac{M}{m},\\
&\int_{\frac{1}{m}}^C |\rho(y)-\brho_m(y)|y^{-2}dy\leq \sum_{j=2}^{\frac{\lceil Cm\rceil}{m}}\int_{\frac{j-1}{m}}^{\frac{j}{m}} \frac{|\rho(y)-\brho_m(y)|}{y^2}dy\leq \sum_{j=2}^{\frac{\lceil Cm\rceil}{m}}\frac{|\rho'(j/m)|}{m}\int_{\frac{j-1}{m}}^{\frac{j}{m}} \frac{dy}{y^2}\leq \frac{2MC}{m}\Big(1-\frac{m}{\lceil Cm\rceil}\Big), \quad \text{using } \eqref{L}\\
&\int_{C}^m |\rho(y)-\brho_m(y)|y^{-2}dy\leq \frac{\sqrt{2 M}e^{-1/2}}{m}\Big(\frac{1}{C}-\frac{1}{m}\Big),\quad \text{using } \eqref{L}\\
&\int_{m}^{\infty} |\rho(y)-\brho_m(y)|\tnu(dy)\leq \int_{m}^{\infty} (\rho(y)+\brho_m(y))\tnu(dy)=2 \tnu([m,\infty)).
\end{align*}
Since the quantities above tend to zero, uniformly in $\nu$, as $m$ goes to infinity, we conclude that Assumption (M3) is satisfied. Finally, condition (M4) fails since
$\int_{|y|\leq 1}|y|\tnu(dy)=2\int_0^1 y^{-1}dy=\infty.$
\end{proof}
\begin{proof}[Proof of Example \ref{ex3}]
Let $\nu$ belong to $\M^{'\alpha,C_1,C_2,M,\varepsilon}$ and define a Lévy measure $\tnu$ having Lévy density (with respect to Lebesgue) given by 
$C_1\frac{e^{-\varepsilon |y|}}{|y|^{1+\alpha}}\I_{y<0}+C_2\frac{e^{-\varepsilon y}}{y^{1+\alpha}}\I_{y>0}$. Hence, 
$\rho(y)=\frac{d\nu}{d\tnu}(y)=e^{-|y|(\lambda_1-\varepsilon)}\I_{y<0}+e^{-y(\lambda_2-\varepsilon)}\I_{y>0}$. Condition (M2) writes: 
$$
C_1\int_{-\infty}^{0}\big(e^{\frac{y}{2}(\lambda_1-\varepsilon)}-1\big)^2\frac{e^{\varepsilon y}}{(-y)^{1+\alpha}}dy+C_2\int_0^{\infty}\big(e^{\frac{y}{2}(\varepsilon-\lambda_2)}
-1\big)^2\frac{e^{-\varepsilon y}}{y^{1+\alpha}}dy<\infty,
$$
which is true since, near zero, the integrands are equivalent to $\frac{1}{|y|^{\alpha-1}}$ and are integrable for $\alpha<1$.
To verify condition (M3) treat again separately the integral $\int_0^{\infty} |\rho(y)-\brho_m(y)|\tnu(dy)$ over the intervals $\pm\Big[0,\frac{1}{m}\Big]$, 
$\pm\Big]\frac{1}{m},m\Big]$, $\pm]m,\infty[$. The integrals between $\frac{1}{m}$ and $m$ and between $-m$ and $-\frac{1}{m}$ are bounded by 
$\max\{C_1,C_2\}(M-\varepsilon)\big(\frac{m^{\alpha-1}}{\alpha}-\frac{1}{\alpha m^{\alpha+1}}\big)$, hence tend to zero. Analogously, the integrals between $-\frac{1}{m}$ and $\frac{1}{m}$ tend to zero thanks to the integrability of $y^{-\alpha}$ in this interval. Finally, condition (M4) follows from
the integrability of $|y|^{-\alpha}$ in $[-1,1]$.
\end{proof}
\vspace{0.2cm}
\bibliographystyle{plain}
\bibliography{refs}
\end{document}